\documentclass[a4paper]{article}

\usepackage{header_article, header_math}

\usepackage{ifthen}

\usepackage[group-separator={,}]{siunitx}
\usepackage{adjustbox}

\usepackage[bottom]{footmisc}

\newcommand\blfootnote[1]{%
  \begingroup
  \renewcommand\thefootnote{}\footnote{#1}%
  \addtocounter{footnote}{-1}%
  \endgroup
}

\usepackage[outline]{contour}
\contourlength{0.1em}

\usepackage{tikz}
\usetikzlibrary{calc}
\usetikzlibrary{patterns}
\usetikzlibrary{decorations.pathreplacing}
\usetikzlibrary{decorations.pathmorphing}
\usetikzlibrary{decorations.markings}
\usetikzlibrary{shapes.geometric}

\tikzset{dot/.style={draw,shape=circle,fill=black,scale=0.4}}
\newcommand{\stwo}{
\draw
    (0,0.5-0.1) to [out=0,in=180]
    (1,0.5) to [out=0,in=90]
    (2,0) to [out=270,in=0]
    (1,-0.5) to [out=180,in=0]
    (0,-0.5+0.1) to [out=180,in=0]
    (-1,-0.5) to [out=180,in=270]
    (-2,0) to [out=90,in=180]
    (-1,0.5) to [out=0,in=180]
    (0,0.5-0.1);
\draw (-1.5,0) to [out=-30,in=180+30] (-0.5,0);
\draw (-1.3,-0.1) to [out=30,in=180-30] (-0.7,-0.1);
\draw (0.5,0) to [out=-30,in=180+30] (1.5,0);
\draw (0.7,-0.1) to [out=30,in=180-30] (1.3,-0.1);
}

\newcommand{\tripod}[3]{
    \draw #1 node [dot] {} -- #2 node [dot] {} -- #3 node [dot] {} -- cycle;
    \draw [red] ($#1!0.5!#2$) .. controls ($0.33*#1+0.33*#2+0.33*#3$) .. ($#2!0.5!#3$);
    \draw [red] ($#2!0.5!#3$) .. controls ($0.33*#1+0.33*#2+0.33*#3$) .. ($#3!0.5!#1$);
    \draw [red] ($#3!0.5!#1$) .. controls ($0.33*#1+0.33*#2+0.33*#3$) .. ($#1!0.5!#2$);
}

\newcommand{\bipod}[3]{
    \draw #1 node [dot] {} -- #2 node [dot] {} -- #3 node [dot] {} -- cycle;
    \draw [red] ($#1!0.5!#2$) .. controls ($0.33*#1+0.33*#2+0.33*#3$) .. ($#2!0.5!#3$);
    \draw [red] ($#2!0.5!#3$) .. controls ($0.33*#1+0.33*#2+0.33*#3$) .. ($#3!0.5!#1$);
}

\usepackage{algorithm}
\usepackage{algpseudocode}

\algnewcommand{\IIf}[1]{\State\algorithmicif\ #1\ \algorithmicthen}
\algnewcommand{\EElse}{\algorithmicelse\ }
\algnewcommand{\EndIIf}{\unskip\ }
\algnewcommand{\algorithmicforeach}{\textbf{for each}}
\algdef{S}[FOR]{ForEach}[1]{\algorithmicforeach\ #1\ \algorithmicdo}

\usepackage{listings}
\definecolor{dkgreen}{rgb}{0,0.6,0}
\definecolor{gray}{rgb}{0.5,0.5,0.5}
\definecolor{mauve}{rgb}{0.58,0,0.82}

\title{Experimental statistics for Mirzakhani's Theorem}
\author{Mark C. Bell}

\begin{document}

\maketitle

\begin{abstract}
In her seminal 2008 paper, Maryam Mirzakhani showed that the ratio that two topological types of curves occur in is a rational number.

In this paper we describe the process by which we obtained experimental evidence that separating and non-separating curves on the surface of genus two occur in the ratio $1 : 48$.
\blfootnote{\href{https://markcbell.github.io}{\url{markcbell.github.io}}}
\end{abstract}

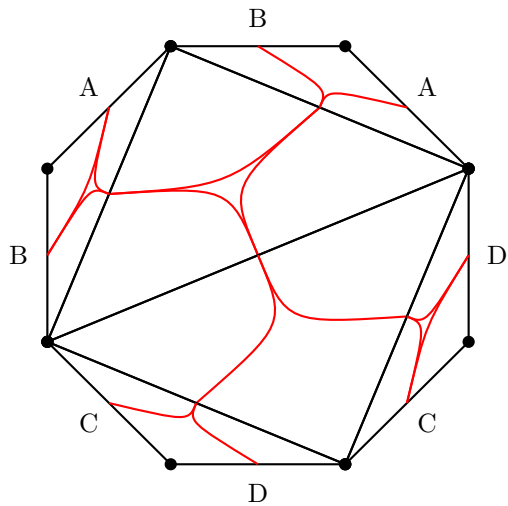
\begin{figure}[htb]
\centering
\begin{tikzpicture}[scale=3, thick]

\begin{scope}[rotate=22.5]
\tripod{(0:1)}{(90:1)}{(180:1)}
\bipod{(45:1)}{(0:1)}{(90:1)}
\tripod{(135:1)}{(90:1)}{(180:1)}
\bipod{(270:1)}{(180:1)}{(0:1)}
\bipod{(225:1)}{(180:1)}{(270:1)}
\tripod{(270:1)}{(315:1)}{(0:1)}
\end{scope}

\foreach \x [count=\i] in {A,B,A,B,C,D,C,D} \node at (45*\i:1.05) {\x};

\end{tikzpicture}
\caption{The train track $\tau_{35}$ on the \texttt{curver} triangulation of $S_{2,1}$.}
\label{fig:train_track}
\end{figure}

\begin{acknowledgements}
The author is grateful to Saul Schleimer for introducing this problem to him and to Vincent Delecroix, Jenya Sapir and Anton Zorich for their many helpful discussions, comments and feedback.

This work used the Extreme Science and Engineering Discovery Environment (XSEDE), which is supported by National Science Foundation (grant number ACI-1548562)~\cite{xsede}.
The author is grateful for the use of
    SuperMIC at Louisiana State University Center for Computation and Technology;
    Comet and Oasis at the San Diego Supercomputer Center; and
    Bridges and Bridges Pylon at the Pittsburgh Supercomputing Center
provided through allocation TG-DMS180008.
\end{acknowledgements}

\section{Introduction}

In her seminal 2008 paper, Maryam Mirzakhani studied the growth rate of simple closed curves on a surface with respect to their length in a hyperbolic metric~\cite{Mirzakhani}.
Her techniques were sufficiently delicate to also extract growth rates for the multicurves of a specific topological type, where the \emph{topological type} of a multicurve $\calL$ is its mapping class group orbit $[\calL] \defeq \Mod(S) \cdot \{\calL\}$.
Importantly, this allowed her to show that the ratio of two of these growth rates is in fact independent of the choice of hyperbolic metric:

\begin{theorem}[{\cite[Corollary~1.4]{Mirzakhani}}]
\label{thrm:mirzakhani}
Let $m$ and $m'$ be topological types of multicurves on a surface $S$.
Choose a hyperbolic metric $X$ on $S$ and let
\[ s_X(L, m) \defeq \{ \textrm{multicurve $\calL$} : \ell_X(\calL) \leq L \inlineand [\calL] = m \}. \]
Then
\[ \lim_{L \to \infty} \frac{|s_X(L, m)|}{|s_X(L, m')|} \]
exists and is a rational number $K(m, m')$ that depends only on $m$ and $m'$ and is independent of the choice of $X$.
\end{theorem}

To prove this, Mirzakhani linked $s_X(L, m)$ to the Weil-Petersson volume of moduli spaces of bordered Riemann surfaces.
She had previously shown the volume of these spaces to be a polynomial function (with coefficients in $\QQ(\pi)$) of their boundary lengths~\cite[Section~5]{MirzakhaniVolume}.
Since these polynomials can be explicitly computed, her process allows a formula for $K(m, m')$ to be obtained.
Within her paper she provided several worked examples of this technique, including computing the distribution of topological types of curves on the (closed) surface of genus two $S_2$~\cite[Section~6]{Mirzakhani}.

Variants of this result were subsequently obtained by others, for example, see the work of Erlandsson--Souto~\cite{ErlandssonSouto}, Rafi--Souto~\cite{RafiSouto} and Rivin~\cite{Rivin},
In particular, Delecroix--Goujard--Zograf--Zorich recently obtained analogous distributions for closed geodesics on square tiled surfaces~\cite{DelecroixGoujardZografZorich}.
However, when using their technique to reproduce Mirzakhani's worked example for the distribution of topological types on $S_2$, they identified two misprints in her example~\cite[Remark~4.15]{DelecroixGoujardZografZorich}.
Together, these appear to affect the $1 : 6$ ratio (and so the probability of separating being $\frac{1}{7}$) calculated by Mirzakhani by a factor of eight (and making the probability of separating $\frac{1}{49}$).

To help rule out the possibility of any other misprints, experimental evidence for this distribution has been obtained by sampling the topological types of a large number of multicurves on $S_2$.
This gives direct evidence that for curves on $S_2$
\[ K(\textrm{non-separating}, \textrm{separating}) = 48 \]
or equivalently that, on a surface of genus two, a long, random, connected, simple, closed geodesic is separating with probability $\frac{1}{49}$.

This paper documents the tools and methods used in this experiment.

\subsection{Experiment}

To obtain an estimate for the distribution of separating and non-separating curves on $S_2$ the following process was used.
See Appendix~\ref{sec:curver} for an example of the code used to achieve this.
\begin{enumerate}
\item Let $\tau_{35}$ be the train track on $S_2$ shown in Figure~\ref{fig:train_track}.
This is shown on the \texttt{curver} triangulation of $S_{2, 1}$ and is named following the binary representation of the normal arcs present in it.
\item A sample of \num{2404171} multicurves was taken with each multicurve being drawn uniformly at random from the set of multicurves that are fully carried by $\tau_{35}$ and with (combinatorial) length at most $L \defeq \num{1000000}$ via the method described in Section~\ref{sec:sample}.
\item Using the techniques of Section~\ref{sec:properties}, the topological type of each of these multicurves was recorded via their canonical name.
\end{enumerate}

The distribution of the \num{1463775} of those that are \emph{primitive}, that is, whose components all lie in distinct isotopy classes, are documented in Table~\ref{tab:S_2_distribution}.
This should be compared to the distribution computed by Delecroix--Goujard--Zograf--Zorich~\cite[Table~1]{DelecroixGoujardZografZorich}, which is also available in Table~\ref{tab:expected_S_2_distribution}.
Of these, we see that \num{613196} are curves (that is, have a single component) --- consisting of \num{12496} that are separating and \num{600700} that are non-separating.
This means that separating and non-separating curves were drawn in a ratio of approximately $1 : 48.07$.

\begin{table}[htb]
\centering
\begin{tabular}{ccc}
Topological Type & Count & Percentage \\ \hline

\adjustbox{margin=4pt}{\begin{tikzpicture}[scale=0.7,thick,baseline=(current bounding box.center)]
\draw [red, dotted] (-1,-0.15) to [out=180,in=180] (-1,-0.5);
\draw [red, dotted] (1,-0.15) to [out=180,in=180] (1,-0.5);
\stwo{}
\draw [red] (-1,-0.15) to [out=0,in=0] (-1,-0.5);
\draw [red] (1,-0.15) to [out=0,in=0] (1,-0.5);
\node at (0, -1) {([0], [\{1,1\}])};
\end{tikzpicture}}
& \num{720649} & 49.23\% \\

\adjustbox{margin=4pt}{\begin{tikzpicture}[scale=0.7,thick,baseline=(current bounding box.center)]
\draw [red, dotted] (-1,-0.15) to [out=180,in=180] (-1,-0.5);
\stwo{}
\draw [red] (-1,-0.15) to [out=0,in=0] (-1,-0.5);
\node at (0, -1) {([1], [\{1\}])};
\end{tikzpicture}}
& \num{600700} & 41.04\% \\

\adjustbox{margin=4pt}{\begin{tikzpicture}[scale=0.7,thick,baseline=(current bounding box.center)]
\draw [red, dotted] (-1,-0.15) to [out=180,in=180] (-1,-0.5);
\draw [red, dotted] (1,-0.15) to [out=180,in=180] (1,-0.5);
\draw [red, dotted] (0,0.4) to [out=180,in=180] (0,-0.4);
\stwo{}
\draw [red] (-1,-0.15) to [out=0,in=0] (-1,-0.5);
\draw [red] (1,-0.15) to [out=0,in=0] (1,-0.5);
\draw [red] (0,0.4) to [out=0,in=0] (0,-0.4);
\node at (0, -1) {([0, 0], [\{1\}, \{1\}, \{1\}])};
\end{tikzpicture}}
& \num{59921} & 4.09\% \\

\adjustbox{margin=4pt}{\begin{tikzpicture}[scale=0.7,thick,baseline=(current bounding box.center)]
\draw [red, dotted] (-1,-0.15) to [out=180,in=180] (-1,-0.5);
\draw [red, dotted] (1,-0.15) to [out=180,in=180] (1,-0.5);
\draw [red, dotted] (-0.7,-0.1) to [out=-30,in=210] (0.7,-0.1);
\stwo{}
\draw [red] (-1,-0.15) to [out=0,in=0] (-1,-0.5);
\draw [red] (1,-0.15) to [out=0,in=0] (1,-0.5);
\draw [red] (-0.7,-0.1) to [out=45,in=135] (0.7,-0.1);
\node at (0, -1) {([0, 0], [\{\}, \{1,1,1\}, \{\}])};
\end{tikzpicture}}
& \num{40053} & 2.74\% \\

\adjustbox{margin=4pt}{\begin{tikzpicture}[scale=0.7,thick,baseline=(current bounding box.center)]
\draw [red, dotted] (-1,-0.15) to [out=180,in=180] (-1,-0.5);
\draw [red, dotted] (0,0.4) to [out=180,in=180] (0,-0.4);
\stwo{}
\draw [red] (-1,-0.15) to [out=0,in=0] (-1,-0.5);
\draw [red] (0,0.4) to [out=0,in=0] (0,-0.4);
\node at (0, -1) {([0, 1], [\{1\}, \{1\}, \{\}])};
\end{tikzpicture}}
& \num{29956} & 2.04\% \\

\adjustbox{margin=4pt}{\begin{tikzpicture}[scale=0.7,thick,baseline=(current bounding box.center)]
\draw [red, dotted] (0,0.4) to [out=180,in=180] (0,-0.4);
\stwo{}
\draw [red] (0,0.4) to [out=0,in=0] (0,-0.4);
\node at (0, -1) {([1, 1], [\{\}, \{1\}, \{\}])};
\end{tikzpicture}}
& \num{12496} & 0.85\%
\end{tabular}
\caption{Sampled distribution of primitive multicurves on $S_{2}$.}
\label{tab:S_2_distribution}
\end{table}

\begin{table}[htb]
\centering
\begin{tabular}{ccc}
Topological Type & Expected Fraction & Expected Percentage \\ \hline
$([0], [\{1,1\}])$ & 288 / 585 & 49.23\% \\
$([1], [\{1\}])$ & 240 / 585 & 41.02\% \\
$([0, 0], [\{1\}, \{1\}, \{1\}])$ & 24 / 585 & 4.10\% \\
$([0, 0], [\{\}, \{1,1,1\}, \{\}])$ & 16 / 585 & 2.74\% \\
$([0, 1], [\{1\}, \{1\}, \{\}])$ & 12 / 585 & 2.05\% \\
$([1, 1], [\{\}, \{1\}, \{\}])$ & 5 / 585 & 0.85\%
\end{tabular}
\caption{Distribution of primitive multicurves on $S_{2}$ derived from the calculations of Delecroix--Goujard--Zograf--Zorich~\cite[Table~1]{DelecroixGoujardZografZorich}.}
\label{tab:expected_S_2_distribution}
\end{table}

\subsection{Preliminaries}

For ease of notations, we will assume throughout that:
\begin{itemize}
\item all multicurves are essential,
\item all polytopes are integral.
\end{itemize}
As usual, we define the \emph{complexity} of:
\begin{itemize}
\item an integer $n$ to be $||n|| \defeq \log(|n|)$,
\item a vector $v$ to be $||v|| \defeq \sum_i ||v[i]||$,
\item a matrix $M$ to be $||M|| \defeq \sum_{i,j} ||M[i, j]||$,
\item a polytope $P = \{v : A \cdot v \geq b\}$ to be $||P|| \defeq ||A|| + ||b||$,
\end{itemize}

\section{Sampling multicurves}
\label{sec:sample}

Inspired by the work of Dunfield--Thurston~\cite{DunfieldThurston}, in this section we describe how to uniformly sample multicurves from a train track.
We use these since they give us a coordinate system that we can use to describe many multicurves with.

\begin{definition}
Suppose that $\tau$ is a train track.
A multicurve $\calL$ is \emph{fully carried by $\tau$} if it is carried by $\tau$ and assigns positive measure to every branch of $\tau$.
We use $s_\tau$ to denote the set of multicurves that are fully carried by $\tau$.
\end{definition}

A train track also gives us a combinatorial notion of length $\ell_\tau(\cdot)$ which records the total weight assigned to the branches of $\tau$ by a fully carried multicurve.

\begin{proposition}
Let $m$ and $m'$ be topological types of multicurves on a surface $S$.
Choose a train track $\tau$ on $S$ that fully carries an open subset of $\ML(S)$ and let
\[ s_\tau(L, m) \defeq \{\calL \in s_\tau : \ell_\tau(\calL) \leq L \inlineand [\calL] = m\}. \]
Then
\[ \lim_{L \to \infty} \frac{|s_\tau(L, m)|}{|s_\tau(L, m')|} = K(m, m') \]
where $K(m, m')$ is the constant of Theorem~\ref{thrm:mirzakhani}. \qed
\end{proposition}

\begin{proof}
First, note that the length function $\ell_\tau$ can be extended to a homogeneous, continuous function $f \from \ML(S) \to \RR_{>0}$.
By the generalisation of Theorem~\ref{thrm:mirzakhani} given by Erlandsson--Souto~\cite{ErlandssonSouto} we then have that
\[ \lim_{L \to \infty} \frac{|s_f(L, m)|}{|s_f(L, m')|} = K(m, m') \]
where $s_f(L, m) \defeq \{\textrm{multicurve $\calL$} : f(\calL) \leq L \inlineand [\calL] = m\}$~\cite[Page~44]{Wright}.

Second, recall that the mapping class group acts ergodically on $\ML(S)$~\cite[Theorem~2]{Masur}.
Therefore, since $\tau$ fully carries an open (and so positive measure) subset of $\ML(S)$, we have that
\[ \lim_{L \to \infty} \frac{|\{\calL \in s_f(L, m) : \textrm{$\calL$ is fully carried by $\tau$} \}|}{|\{\calL \in s_f(L, m') : \textrm{$\calL$ is fully carried by $\tau$} \}|}
= \lim_{L \to \infty} \frac{|s_f(L, m)|}{|s_f(L, m')|} \]
However, the left hand term is precisely the limit that we are interested in.
\end{proof}

Hence, in the case of $S_2$ we will observe the same distribution of topological types of multicurves by sampling from $s_{\tau_{35}}$

\subsection{Sampling from a train track}

Throughout this section, suppose that $\tau$ is a train track with branches labelled $1, \ldots, m$.
At each switch $s$, the branches that meet it are grouped into two gates.
Let $s^+$ and $s^-$ denote the set of branch labels of these each gate.
If $\calL$ is a multicurve that is fully carried by $\tau$ and assigns weights $x_1, \ldots, x_m$ to the branches of $\tau$ then for every switch $s$ these variables satisfy the switch condition:
\[ \sum_{i \in s^+} x_i = \sum_{i \in s^-} x_i. \]
Furthermore, since $\calL$ is fully carried, we have that $x_i \geq 1$ for every $i$.
Together these equations and inequalities define a polytope $P_\tau$ in $\RR^m$ with a natural bijection between the integral points within it and the multicurves in $s_\tau$.

Let $P_\tau(L)$ be the polytope obtained by intersecting $P_\tau$ with the halfspace $\{x_1 + \cdots + x_m \leq L\}$.
Again there is a natural bijection between integral points in $P_\tau(L)$ and multicurves in
\[ s_\tau(L) \defeq \{\calL \in s_\tau : \ell_\tau(\calL) \leq L\}. \]

\begin{algorithm}[htb!]
\caption{Get integral point}
\label{alg:get_integral_point}
\begin{algorithmic}
\Require{A compact polytope $P$ and integer $i$}
\Ensure{The $i$\nth{} integral point in $P$}
\State $D \gets$ the ambient dimension of $P$
\State $p \gets (0, 0, \ldots, 0) \in \ZZ^D$
\State $l_0 \gets -2^{||P||}, u_0 \gets 2^{||P||}$  \Comment{So that $l_0 \leq p[d] \leq u_0$ for every $d$.}
\For {$d \gets 1, \ldots, D$}
    \State $l \gets l_0, u \gets u_0$
    \While {$l < u - 1$}
        \State $g \gets \lfloor \frac{1}{2} (l + u) \rfloor$  \Comment{Use the midpoint for the next guess.}
        \If {$i < c(P \cap \{x_d < g\})$}  \Comment{Use Barvinok's algorithm.}
            \State $u \gets g$
            \State $P \gets P \cap \{x_d < g\}$
        \Else
            \State $l \gets g$
            \State $i \gets i - c(P \cap \{x_d < g\})$  \Comment{Use Barvinok's algorithm.}
            \State $P \gets P \cap \{x_d \geq g\}$
        \EndIf
    \EndWhile
    \State $p[d] \gets l$
\EndFor
\State \Return $p$
\end{algorithmic}
\end{algorithm}

Recall that Barvinok's algorithm can determine the number of integral points $|P|$ within a given polytope $P$~\cite{Barivnok}.
When the dimension of the polytope is fixed, the algorithm runs in polynomial time on the size of the input~\cite[Theorem~1.2]{Barivnok}.
Hence, when the underlying surface is fixed, $|s_\tau(L)| = |P_\tau(L)|$ can be computed in $\poly(||L||)$ time.

Furthermore, if we order the integral points $p_1, \ldots, p_k$ within a polytope $P$ lexicographically then Barvinok's algorithm also allows us to efficiently determine the coordinates of $p_i$.
To do this we tackle each coordinate of $p_i$ in turn.
We choose a hyperplane that splits $P$ in half and use Barvinok's algorithm to determine which side of this hyperplane $p_i$ lies on.
We intersect $P$ with the correct halfspace and repeat until we reach the polytope $\{p_i\}$ as described in Algorithm~\ref{alg:get_integral_point}.
By choosing these hyperplanes to bisect $P$ we can ensure that no more than $\dim_{\textrm{amb}}(P) \cdot ||P||$ rounds are needed.
In the case of $P_\tau(L)$, no more than $-6 \chi(S) \cdot ||L||$ rounds are needed.
This algorithm was implemented as part of SageMath 8.2~\cite{sagemath} and is where the majority of the processing time is spent.

Therefore by choosing an index $i$ uniformly at random from $[1, |s_\tau(L)|]$ and using Algorithm~\ref{alg:get_integral_point} we can efficiently obtain a multicurve chosen uniformly at random from $s_\tau(L)$.
See Algorithm~\ref{alg:random_multicurve}.

\begin{algorithm}[ht]
\caption{Random multicurves}
\label{alg:random_multicurve}
\begin{algorithmic}
\Require{A train track $\tau$ and integer $L$}
\Ensure{A multicurve chosen uniformly at random from $s_\tau(L)$}
    \State $i \gets$ an integer chosen uniformly at random from $[1, |s_\tau(L)|]$
    \State $p \gets$ the $i$\nth{} integral point of $P_\tau(L)$  \Comment{Use Algorithm~\ref{alg:get_integral_point}.}
    \State $\calL \gets$ the multicurve in $s_\tau(L)$ corresponding to $p$
    \State \Return $\calL$
\end{algorithmic}
\end{algorithm}

\section{Determining properties of multicurves}
\label{sec:properties}

In this section we will describe how to efficiently determine properties, including topological type, of a multicurve.
We will begin by showing how to do this on a punctured surface $S$ before showing how to extend this process to closed surfaces.

\begin{definition}
Suppose that $\calL$ is a multicurve on a punctured surface $S$.
An ideal triangulation $\calT$ is \emph{$\calL$--short} if for each component $\gamma$ of $\calL$ there is an edge $e$ of $\calT$ that is \emph{parallel} to it.
That is, $\gamma$ is a component of $\partial N(e)$.
For example, see Figure~\ref{fig:parallel}.
\end{definition}

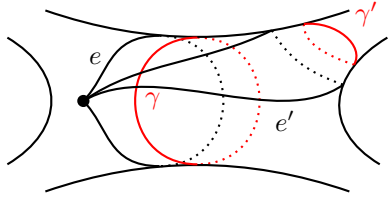
\begin{figure}[htb]
\centering
\begin{tikzpicture}[xscale=2, yscale=1.2, thick]

\draw [dotted] (-0.25, 0.72) to[out=0, in=0] (-0.25, -0.72);
\draw [dotted, red] (0, 0.7) to[out=0, in=0] (0, -0.7);

\draw [dotted] (0.49, 0.78) to [out=290, in=150] (0.96, 0.18);
\draw [dotted, red] (0.7, 0.85) to[out=270, in=150] (1.03, 0.4);

\draw (-1, 1) to [out=-30,in=210] (1, 1);
\draw (-1, -1) to [out=30,in=150] (1, -1);

\draw (-1.25, 0.7) to [out=-45,in=45] (-1.25, -0.7);
\draw (1.25, 0.7) to [out=225,in=135] (1.2, -0.7);

\draw [red] (0, 0.7) to[out=180, in=180] node [right] {$\gamma$} (0, -0.7);
\draw (-0.75, 0) to[out=60,in=180] node [left] {$e$} (-0.25, 0.72);
\draw (-0.75, 0) to[out=300,in=180] (-0.25, -0.72);

\draw [red] (0.7, 0.85) to[out=10, in=70] node [above right] {$\gamma'$} (1.03, 0.4);
\draw (-0.75, 0) to [out=50, in=220] (0.49, 0.78);
\draw (-0.75, 0) to [out=45, in=230] node [near end, below] {$e'$} (0.96, 0.18);

\node [dot] at (-0.75, 0) {};

\end{tikzpicture}
\caption{Edges $e$ and $e'$ parallel to components $\gamma$ and $\gamma'$ of a multicurve.}
\label{fig:parallel}
\end{figure}

It is very easy to compute many properties of $\calL$ when it is given via its intersection numbers with the edges of a $\calL$--short triangulation.
For example, how many components it has, how many of its components lie in the same isotopy class or its geometric intersection number with another multicurve $\calL'$~\cite{curver}~\cite{BellSimplifying}.
Fortunately, it is also straightforward to move from any triangulation to a short one.

\begin{theorem}[{\cite[Section~4]{AgolHassThurston},~\cite[Theorem~3.7]{BellSimplifying},~\cite[Section~6.3]{EricksonNayyeri}}]
\label{thrm:shorten}
Suppose that $\calL$ is a multicurve on $\calT$.
There is an algorithm that computes a sequence of moves (flips and twists) from $\calT$ to a $\calL$--short triangulation that runs in $\poly(||\calT(\calL)||)$ time. \qed
\end{theorem}

This algorithm has been implemented as part of \texttt{curver}~\cite{curver}.

\begin{remark}
If a property occurs for a definite percentage of multicurves and it can be efficiently computed by this approach then we can also efficiently condition our distribution of multicurves on it.
Such properties include:
\begin{itemize}
\item having a specific number of components (of which being connected is a special case) and
\item each of its components lying in a distinct isotopy class.
\end{itemize}
\end{remark}

\subsection{Topological Type}
\label{sub:topological_type}

We now describe how to compute and record the topological type of a multicurve $\calL$ via a labelled graph.
By Theorem~\ref{thrm:shorten}, without loss of generality we may assume throughout that $\calL$ is given on an $\calL$--short triangulation.

\begin{figure}[htb]
\centering
\begin{tikzpicture}[scale=1.25,thick]

\begin{scope}[shift={(-2.5,0)}]
\draw [red, dotted] (-1,-0.15) to [out=180,in=180] (-1,-0.5);
\draw [red, dotted] (0,-0.4) to [out=180,in=180] (0,0.4);
\draw [red, dotted] (0.1,-0.4) to [out=180,in=180] (0.1,0.4);
\draw
	(0,0.5-0.1) to [out=0,in=180]
	(1,0.5) to [out=0,in=90]
	(2,0) to [out=270,in=0]
	(1,-0.5) to [out=180,in=0]
	(0,-0.5+0.1) to [out=180,in=0]
	(-1,-0.5) to [out=180,in=270]
	(-2,0) to [out=90,in=180]
	(-1,0.5) to [out=0,in=180]
	(0,0.5-0.1);
\draw (-1.5,0) to [out=-30,in=180+30] (-0.5,0);
\draw (-1.3,-0.1) to [out=30,in=180-30] (-0.7,-0.1);
\draw (0.5,0) to [out=-30,in=180+30] (1.5,0);
\draw (0.7,-0.1) to [out=30,in=180-30] (1.3,-0.1);
\draw [red] (-1,-0.15) to [out=0,in=0] (-1,-0.5);
\draw [red] (0,-0.4) to [out=0,in=0] node [left] {$\calL$} (0,0.4);
\draw [red] (0.1,-0.4) to [out=0,in=0] (0.1,0.4);
\node [dot] at (1.75,0) {};
\node [dot] at (-1.75,0.1) {};
\node [dot] at (-1.75,-0.1) {};
\node at (0,-0.75) {$S$};
\end{scope}

\begin{scope}[shift={(2.5,0)}]
\draw
	(-1.5,0) to [out=-30,in=90]
	(-1.1,-0.35) to [out=270,in=0]
	(-1.25,-0.5) to [out=180,in=270]
	(-2,0) to [out=90,in=180]
	(-1,0.5) to [out=0,in=90]
	(0, 0) to [out=270,in=0]
	(-0.75,-0.5) to [out=180,in=270]
	(-0.9,-0.35) to [out=90,in=210]
	(-0.5,0);

\draw
	(0.5,0) to [out=90,in=180]
	(1.5,0.5) to [out=0,in=90]
	(2.5,0) to [out=270,in=0]
	(1.5,-0.5) to [out=180,in=270]
	(0.5,0);

\draw (-1.3,-0.1) to [out=30,in=180-30] (-0.7,-0.1);
\draw (1,0) to [out=-30,in=180+30] (2,0);
\draw (1.2,-0.1) to [out=30,in=180-30] (1.8,-0.1);
\node [dot] at (-0.25,0) {};
\node [dot] at (0.75,0) {};
\node [dot] at (2.25,0) {};
\node [dot] at (-1.25,-0.3) {};
\node [dot] at (-0.75,-0.3) {};
\node [dot] at (-1.75,0.1) {};
\node [dot] at (-1.75,-0.1) {};
\node at (0.25,-0.75) {$S'$};
\end{scope}

\draw [->] (-0.35,0) -- node [above] {Crush} (0.35,0);

\end{tikzpicture}
\caption{Crushing along a multicurve $\calL$.}
\label{fig:crush}
\end{figure}
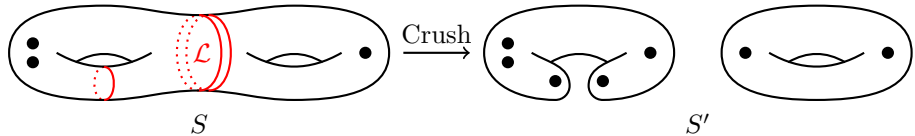

Let $S'$ be the surface obtained by \emph{crushing} $S$ along $\calL$.
For example, see Figure~\ref{fig:crush}.
Since $\calL$ is given on an $\calL$--short triangulation $\calT$, it is straightforward to compute an ideal triangulation $\calT'$ of $S'$.
For each component $\gamma$ of $\calL$, we \emph{extract} a triangle meeting $\gamma$ and an edge $e$ that is parallel to $\gamma$, as shown in Figure~\ref{fig:crush_triangulation}

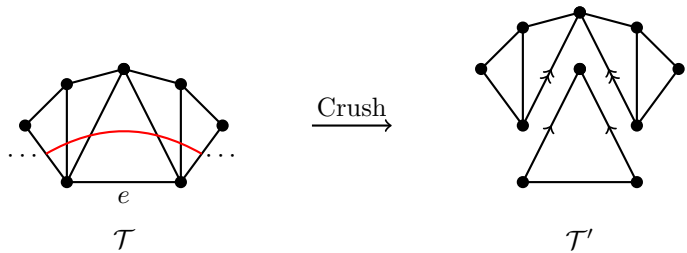
\begin{figure}[htb]
\centering
\begin{tikzpicture}[scale=1.5,thick]

\begin{scope}[shift={(-2,-0.5)}]
\draw (-0.5, 0) node [dot] {} to node [below] {$e$} (0.5, 0) node [dot] {};

\node at (165:0.9) {$\cdots$};

\draw (-0.5, 0) to (150:1) node [dot] {};
\draw (-0.5, 0) to (120:1) node [dot] {};
\draw (-0.5, 0) to (90:1) node [dot] {};

\draw (0.5, 0) to (90:1) node [dot] {};
\draw (0.5, 0) to (60:1) node [dot] {};
\draw (0.5, 0) to (30:1) node [dot] {};

\node at (15:0.9) {$\cdots$};

\draw (150:1) -- (120:1) -- (90:1) -- (60:1) -- (30:1);

\draw [red] ($(-0.5,0)!0.5!(150:1)$) to [out=30,in=150] ($(0.5,0)!0.5!(30:1)$);
\node at (0, -0.5) {$\calT$};

\end{scope}

\begin{scope}[shift={(2,0)}]
\draw (-0.5, -0.5) node [dot] {} to (0.5, -0.5) node [dot] {};
\draw[decoration={markings,mark=at position 0.5 with {\arrow{>}}}, postaction={decorate}] (-0.5, -0.5) to (0, 0.5) node [dot] {};
\draw[decoration={markings,mark=at position 0.5 with {\arrow{>}}}, postaction={decorate}] (0.5, -0.5) to (0, 0.5) node [dot] {};

\draw (-0.5, 0) node [dot] {} to (150:1) node [dot] {};
\draw (-0.5, 0) to (120:1) node [dot] {};
\draw[decoration={markings,mark=at position 0.5 with {\arrow{>>}}}, postaction={decorate}] (-0.5, 0) to (90:1) node [dot] {};

\draw[decoration={markings,mark=at position 0.5 with {\arrow{>>}}}, postaction={decorate}] (0.5, 0) to (90:1) node [dot] {};
\draw (0.5, 0) node [dot] {} to (60:1) node [dot] {};
\draw (0.5, 0) to (30:1) node [dot] {};

\draw (150:1) -- (120:1) -- (90:1) -- (60:1) -- (30:1);

\node at (0, -1) {$\calT'$};
\end{scope}

\draw [->] (-0.35,0) -- node [above] {Crush} (0.35,0);

\end{tikzpicture}
\caption{Crushing a triangulation along a component of a multicurve.}
\label{fig:crush_triangulation}
\end{figure}

For ease of notation let $\{S'_i\}$ be the connected components of $S'$ and let $g_i$ denote the genus of $S'_i$.
With $\calT'$ in hand, $g_i$ can be computed directly from the number of vertices and edges in each connected component.
When $\calL$ is a curve, $S'$ consists of either one or two components and their topological type determines the topological type of $\calL$.
However when $\calL$ is a multicurve we are required to record additional information.

\begin{definition}
The \emph{partition graph} of $\calL$ is the labelled graph $G(\calL)$ with:
\begin{itemize}
\item a vertex corresponding to each $S'_i$ that is labelled $g_i$ and
\item a half-edge emanating from each vertex for each puncture of the corresponding surface $S'_i$.
These are labelled with the multiplicity of the corresponding component in the original surface.
Two half-edges are connected if and only if their punctures came from crushing along the same component.
\end{itemize}
We obtain its \emph{completion} $\hat{G}(\calL)$ by adding a dummy vertex labelled $-1$ and connecting all hanging half-edges to it.
For example, see Figure~\ref{fig:partition_graph}.
\end{definition}

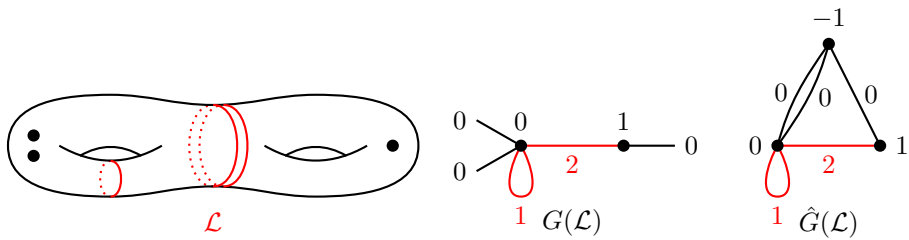
\begin{figure}[htb]
\centering
\begin{tikzpicture}[scale=1.35,thick]

\begin{scope}[shift={(-1.5,0)}]
\draw [red, dotted] (-1,-0.15) to [out=180,in=180] (-1,-0.5);
\draw [red, dotted] (0,-0.4) to [out=180,in=180] (0,0.4);
\draw [red, dotted] (0.1,-0.4) to [out=180,in=180] (0.1,0.4);
\draw
	(0,0.5-0.1) to [out=0,in=180]
	(1,0.5) to [out=0,in=90]
	(2,0) to [out=270,in=0]
	(1,-0.5) to [out=180,in=0]
	(0,-0.5+0.1) to [out=180,in=0]
	(-1,-0.5) to [out=180,in=270]
	(-2,0) to [out=90,in=180]
	(-1,0.5) to [out=0,in=180]
	(0,0.5-0.1);
\draw (-1.5,0) to [out=-30,in=180+30] (-0.5,0);
\draw (-1.3,-0.1) to [out=30,in=180-30] (-0.7,-0.1);
\draw (0.5,0) to [out=-30,in=180+30] (1.5,0);
\draw (0.7,-0.1) to [out=30,in=180-30] (1.3,-0.1);
\draw [red] (-1,-0.15) to [out=0,in=0] (-1,-0.5);
\draw [red] (0,-0.4) to [out=0,in=0] (0,0.4);
\draw [red] (0.1,-0.4) to [out=0,in=0] (0.1,0.4);

\node [dot] at (1.75,0) {};
\node [dot] at (-1.75,0.1) {};
\node [dot] at (-1.75,-0.1) {};

\node [red] at (0,-0.75) {$\calL$};
\end{scope}

\begin{scope}[shift={(1.5,0)}]

\draw [red] (0,0) -- node [below] {$2$} (1,0);
\draw [red] (0,0) to [out=-60,in=0] (0,-0.5) node [below] {$1$} to [out=180,in=-120] (0,0);
\draw (0,0) -- (150:0.5) node [left] {$0$};
\draw (0,0) -- (210:0.5) node [left] {$0$};
\draw (1,0) -- (1.5,0) node [right] {$0$};

\node [dot, label={above:$0$}] at (0,0) {};
\node [dot, label={above:$1$}] at (1,0) {};

\node at (0.5,-0.75) {$G(\calL)$};

\end{scope}

\begin{scope}[shift={(4,0)}]
\draw [red] (0,0) -- node [below] {$2$} (1,0);
\draw [red] (0,0) to [out=-60,in=0] (0,-0.5) node [below] {$1$} to [out=180,in=-120] (0,0);
\draw (0,0) to [out=75,in=235] node [left] {$0$} (0.5,1);
\draw (0,0) to [out=55,in=255] node [right] {$0$} (0.5,1);
\draw (1,0) -- node [right] {$0$} (0.5,1);

\node [dot, label={left:$0$}] at (0,0) {};
\node [dot, label={right:$1$}] at (1,0) {};
\node [dot, label={above:$-1$}] at (0.5,1) {};

\node at (0.5,-0.75) {$\hat{G}(\calL)$};
\end{scope}

\end{tikzpicture}
\caption{A multicurve, its partition graph and its completion.}
\label{fig:partition_graph}
\end{figure}

\begin{proposition}
The following are equivalent:
\begin{itemize}
\item multicurves $\calL$ and $\calL'$ have the same topological type,
\item $G(\calL)$ and $G(\calL')$ are (label) isomorphic graphs, and
\item $\hat{G}(\calL)$ and $\hat{G}(\calL')$ are (label) isomorphic graphs. \qed
\end{itemize}
\end{proposition}

Although there are significantly more efficient methods for solving the graph isomorphism problem, since the graphs we will encounter are so small, we choose to simply write each of them down in a canonical form.

Now note that if we choose an ordering for its vertices then we may describe a labelled graph via a pair consisting of:
\begin{enumerate}
\item the list of vertex labels (in order),
\item a matrix where the $i,j$\nth{} entry is the multiset of edge labels that connect from vertex $i$ to vertex $j$
\end{enumerate}
We use the ordering of the vertices that produces the lexicographically smallest pair and refer to this pair as the \emph{canonical pair}.
Using this, two graphs are label isomorphic if and only if their canonical pairs are equal.
Finally, for convenience, since the matrix is symmetric, we omit the entries below the diagonal and flatten it into a list of multisets.
We refer to this as the \emph{canonical name} of the graph.
For example, the canonical pair for the completed partition graph $\hat{G}(\calL)$ shown in Figure~\ref{fig:partition_graph} is:
\[ ([-1, 0, 1], \left(\begin{matrix}
    \{\} & \{0, 0\} & \{0\} \\
    \{0, 0\} & \{1\} & \{2\} \\
    \{0\} & \{2\} & \{\}
\end{matrix}\right)) \]
and so its canonical name is:
\[ ([-1, 0, 1], [\{\}, \{0, 0\}, \{0\}, \{1\}, \{2\}, \{\}]). \]

\subsection{Closed surfaces}
\label{sub:closed}

If $\calL$ is a multicurve on a closed surface $S$ then we may introduce an artificial puncture and consider an inclusion $\mathring{\calL}$ of $\calL$ onto the punctured surface $\mathring{S}$.
By removing the dummy vertex of $\hat{G}(\mathring{\calL})$, or equivalently by dropping all of the hanging half-edges of $G(\mathring{\calL})$, we obtain a graph $G(\calL)$ encoding the topological type of $\calL$ on $S$.
Again we describe such a graph by its canonical name.
For example, see Figure~\ref{fig:closed_partition_graph} whose canonical name is:
\[ ([0, 1], [\{1\}, \{2\}, \{\}]). \]

\begin{figure}[htb]
\centering
\begin{tikzpicture}[scale=1.35,thick]

\begin{scope}[shift={(-1.5,0)}]
\draw [red, dotted] (-1,-0.15) to [out=180,in=180] (-1,-0.5);
\draw [red, dotted] (0,-0.4) to [out=180,in=180] (0,0.4);
\draw [red, dotted] (0.1,-0.4) to [out=180,in=180] (0.1,0.4);
\draw
	(0,0.5-0.1) to [out=0,in=180]
	(1,0.5) to [out=0,in=90]
	(2,0) to [out=270,in=0]
	(1,-0.5) to [out=180,in=0]
	(0,-0.5+0.1) to [out=180,in=0]
	(-1,-0.5) to [out=180,in=270]
	(-2,0) to [out=90,in=180]
	(-1,0.5) to [out=0,in=180]
	(0,0.5-0.1);
\draw (-1.5,0) to [out=-30,in=180+30] (-0.5,0);
\draw (-1.3,-0.1) to [out=30,in=180-30] (-0.7,-0.1);
\draw (0.5,0) to [out=-30,in=180+30] (1.5,0);
\draw (0.7,-0.1) to [out=30,in=180-30] (1.3,-0.1);
\draw [red] (-1,-0.15) to [out=0,in=0] (-1,-0.5);
\draw [red] (0,-0.4) to [out=0,in=0] (0,0.4);
\draw [red] (0.1,-0.4) to [out=0,in=0] (0.1,0.4);

\node [red] at (0,-0.75) {$\calL$};
\end{scope}

\begin{scope}[shift={(1.5,0)}]

\draw [red] (0,0) -- node [below] {$2$} (1,0);
\draw [red] (0,0) to [out=-60,in=0] (0,-0.5) node [below] {$1$} to [out=180,in=-120] (0,0);

\node [dot, label={above:$0$}] at (0,0) {};
\node [dot, label={above:$1$}] at (1,0) {};

\node at (0.5,-1.) {$G(\calL) = \hat{G}(\calL)$};

\end{scope}

\end{tikzpicture}
\caption{A multicurve and (the completion of) its partition graph.}
\label{fig:closed_partition_graph}
\end{figure}
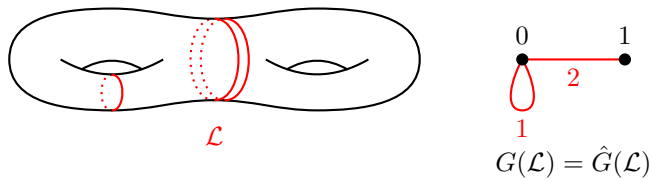

Again this extra step has been implemented within \texttt{curver}.

\begin{remark}
A similar technique can be used to record the topological type of any integral lamination, including ones in which some or all of the components are arcs.
We say that an edge $e$ is parallel to an arc $\gamma$ if they are equal.
\end{remark}

\section{Other work}

The processes described within the paper are completely general.
They can be used to obtain experimental evidence for $K(m, m')$ for any pair of topological types of multicurves.

\subsection{Six-times punctured sphere}

We have also used this approach to obtain exploratory data about the distribution of topological types on the six-time-punctured sphere $S_{0, 6}$.
In this case there are two topological types of curve and representatives $\gamma_{3,3}$ and $\gamma_{2,4}$ (named for how they separate the punctures of $S_{0,6}$) are shown in Figure~\ref{fig:S_0_6}.
These have canonical names
\[ ([-1, 0, 0], [\{\}, \{0,0,0\}, \{0,0,0\}, \{\}, \{1\}, \{\}]) \]
and
\[ ([-1, 0, 0], [\{\}, \{0,0\}, \{0,0,0,0\}, \{\}, \{1\}, \{\}]) \]
respectively.

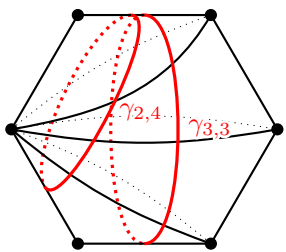
\begin{figure}[htb]
\centering



\begin{tikzpicture}[scale=1.75, thick]
\begin{scope}[xshift=-35]
\draw [red, dotted, very thick] ($(60:1)!0.5!(120:1)$) to [out=180, in=180, looseness=0.5] ($(240:1)!0.5!(300:1)$);
\draw [red, dotted, very thick] (-0.1, 0.865) to [out=180, in=120, looseness=0.5] ($(180:1)!0.5!(240:1)$);

\draw [dotted, thin] (180:1) to [out=40, in=200] (60:1);
\draw [dotted, thin] (180:1) to [out=10, in=170] (0:1);
\draw [dotted, thin] (180:1) to [out=-20, in=140] (300:1);

\foreach \t in {0, 60, ..., 300}
    \draw (\t:1) to (\t+60:1) node [dot] {};

\draw (180:1) to [out=10, in=240] (60:1);
\draw (180:1) to [out=-10, in=190] (0:1);
\draw (180:1) to [out=-40, in=160] (300:1);
\draw [red, very thick] ($(60:1)!0.5!(120:1)$) to [out=0, in=0, looseness=0.5] node [right] {\contour*{white}{$\gamma_{3,3}$}} ($(240:1)!0.5!(300:1)$);
    \draw [red, very thick] (-0.1, 0.865) to [out=0, in=300, looseness=0.5] node [right] {\contour*{white}{$\gamma_{2,4}$}} ($(180:1)!0.5!(240:1)$);
\end{scope}

\end{tikzpicture}
\caption{Curves on $S_{0, 6}$.}
\label{fig:S_0_6}
\end{figure}

A similar experiment involving \num{61655} samples was performed in this case.
Of these \num{35115} had type $[\gamma_{3,3}]$ while \num{26540} had type $[\gamma_{2,4}]$.
Suggesting that on $S_{0, 6}$
\[ K([\gamma_{3,3}], [\gamma_{2,4}]) = \frac{4}{3}. \]

\subsection{Twice-punctured torus}

Similarly, in the case of the twice-punctured torus $S_{1, 2}$ any curve is either separating or non-separating.
These have canonical names
\[ ([-1, 0, 1], [\{\}, \{0,0\}, \{\}, \{\}, \{1\}, \{\}]) \inlineand ([-1, 0], [\{\}, \{0,0\}, \{1\}]) \]
respectively.

\begin{figure}[htb]
\centering
\begin{tikzpicture}[scale=1.5,thick]

\draw [red, dotted] (-1,-0.15) to [out=180,in=180] (-1,-0.5);
\draw
	(0,0) to [out=270,in=0]
	(-1,-0.5) to [out=180,in=270]
	(-2,0) to [out=90,in=180]
	(-1,0.5) to [out=0,in=90]
	(0,0);
\draw (-1.5,0) to [out=-30,in=180+30] (-0.5,0);
\draw (-1.3,-0.1) to [out=30,in=180-30] (-0.7,-0.1);
\draw [red] (-1,-0.15) to [out=0,in=0] (-1,-0.5);
\draw [red] (-0.3,-0.3) to [out=0,in=0] (-0.3,0.3) to [out=180,in=180] (-0.3,-0.3);

\node [dot] at (-0.3, 0.2) {};
\node [dot] at (-0.3, -0.2) {};

\end{tikzpicture}
\caption{Curves on $S_{1, 2}$.}
\label{fig:S_1_2}
\end{figure}
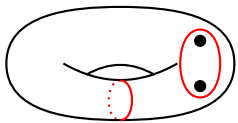

A similar experiment involving \num{118140} samples was performed in this case.
Of these \num{113412} were separating while \num{4728} were non-separating.
Suggesting that on $S_{1, 2}$
\[ K(\textrm{non-separating}, \textrm{separating}) = 24 \]

\appendix

\clearpage
\section{Curver}
\label{sec:curver}

The following Python script uses \texttt{curver} to sample multicurves from $\tau_{35}$ and log their topological type.
Note the use of \texttt{cddlib} is only for performance.

\begin{lstlisting}
# Requires Sage 8.2+ and latte-int, which is installable via the command:
# > sage -i latte_int

from sage.all import Polyhedron
from random import randint

import curver

max_weight = 1000000

T = curver.load(2, 1).triangulation
zeroed = [0,0,0,0,0,0,0,0,0,0,0,0,1,0,0,0,1,1]  # 35 in binary.

# Build the polytope.
eqns, ieqs = [], []
# Edge equations.
for i in range(T.zeta):
    eqn = [0] * 2*T.zeta
    # Since all edges of T are flippable.
    A, B = T.corner_lookup[i], T.corner_lookup[~i]
    x, y = A.labels[1], A.labels[2]
    z, w = B.labels[1], B.labels[2]
    eqn[x], eqn[y], eqn[z], eqn[w] = +1, +1, -1, -1
    eqns.append([0] + eqn)  # V_x + X_y == V_z + V_w.
# Zeroed (in)equalities
for i in range(2*T.zeta):
    if not zeroed[i]:
        ieq = [0] * 2*T.zeta
        ieq[i] = +1
        ieqs.append([-1] + ieq)  # V_i >= 1.
    else:  # Zeroed equation.
        eqn = [0] * 2*T.zeta
        eqn[i] = +1
        eqns.append([0] + eqn)  # V_i == 0.
# Max weight inequality.
ieqs.append([max_weight] + [-1] * 2*T.zeta)  # sum V_i <= max_weight.
P = Polyhedron(eqns=eqns, ieqs=ieqs)

num_integral_points = P.integral_points_count(triangulation='cddlib')
for _ in range(0):
    index = randrange(0, num_integral_points)
    branch_weights = [int(x) for x in P.get_integral_point(index, triangulation='cddlib')]
    triangulation_weights = [sum(branch_weights[label] for label in T.corner_lookup[i].labels[1:]) for i in range(T.zeta)]
    multicurve = T(triangulation_weights)
    print('{}: {} component(s), type {}'.format(index, multicurve.num_components(), multicurve.topological_type(closed=True)))

\end{lstlisting}

The same can be achieved without SageMath via the docker command:

\begin{lstlisting}
docker run --rm -t markcbell/mirzakhani
\end{lstlisting}

\clearpage
\bibliographystyle{amsplain}
\bibliography{bibliography}

\end{document}